\documentclass[12pt]{amsart}
\usepackage{graphicx}
\usepackage[T1]{fontenc}
\usepackage{color}

\newtheorem{thm}{Theorem}[section]

\newtheorem{prop}[thm]{Proposition}

\theoremstyle{definition}
\newtheorem{defi}[thm]{Definition}
\newtheorem{example}[thm]{Example}
\newtheorem{examples}[thm]{Examples}

\theoremstyle{remark}
\newtheorem{remark}[thm]{Remark}

\numberwithin{equation}{section}

\newcommand{\dd}{\mathrm{d}}

\newcommand{\rr}{{\mathbb R}}
\newcommand{\nat}{{\mathbb N}}
\newcommand{\ganz}{{\mathbb Z}}

\newcommand{\tim}{{\mathbb T}}
\newcommand{\G}{{\mathbb G}}

\newcommand{\RV}{\operatorname{RV}}

\newcommand{\eqfd}{\stackrel{\rm fd}{=}}
\newcommand{\fdc}{\stackrel{\rm d}{\longrightarrow}}
\newcommand{\fdd}{\stackrel{\rm fd}{\longrightarrow}}
\newcommand{\Exp}{\mathbb E}

\allowdisplaybreaks

\begin{document}

\sloppy

\title{Dilatively semistable stochastic processes}

\author{Peter Kern}
\address{Peter Kern, Mathematical Institute, Heinrich-Heine-University D\"usseldorf, Universit\"atsstr.\ 1, D-40225 D\"usseldorf, Germany}
\email{kern\@@{}math.uni-duesseldorf.de}

\author{Lina Wedrich}
\address{Lina Wedrich, Mathematical Institute, Heinrich-Heine-University D\"usseldorf, Universit\"atsstr.\ 1, D-40225 D\"usseldorf, Germany}
\email{wedrich\@@{}math.uni-duesseldorf.de}

\thanks{This work is kindly supported by the German Academic Exchange Service (DAAD) financed by funds of the Federal Ministry of Education and Research (BMBF) under project no. 57059326}

\date{\today}

\begin{abstract}
Dilative semistability extends the notion of semi-selfsimilarity for infinitely divisible stochastic processes by introducing an additional scaling in the convolution exponent. It is shown that this scaling relation is a natural extension of dilative stability and some examples of dilatively semistable processes are given. We further characterize dilatively stable and dilatively semistable processes as limits for certain rescaled aggregations of independent processes.
\end{abstract}



\maketitle

\baselineskip=18pt

\section{Introduction}

Let $\tim$ be either $\rr$, $[0,\infty)$ or $(0,\infty)$. Following \cite{BPK} a stochastic process $(X_t)_{t\in\tim}$ on $\rr$ is called $(\alpha,\delta)$-{\it dilatively stable} for some parameters $\alpha,\delta\in\rr$ if all its finite-dimensional marginal distributions are infinitely divisible and the scaling relation
$$\psi_{Tt_1,\ldots,Tt_k}(\theta_1,\ldots,\theta_k)=T^\delta\psi_{t_1,\ldots,t_k}(T^{\alpha-\delta/2}\theta_1,\ldots,T^{\alpha-\delta/2}\theta_1)$$
holds for all $T>0$, $k\in\nat$, $\theta_1,\ldots,\theta_k\in\rr$, and $t_1,\ldots,t_k\in\tim$, where $\psi_{t_1,\ldots,t_k}$ denotes the log-characteristic function of $(X_{t_1},\ldots,X_{t_k})$, which is the unique continuous function with $\psi_{t_1,\ldots,t_k}(0,\ldots,0)=0$ fulfilling
$$\Exp\Big[\exp\Big(i\sum_{j=1}^k\theta_jX_{t_j}\Big)\Big]=\exp\left(\psi_{t_1,\ldots,t_k}(\theta_1,\ldots,\theta_k)\right).$$
This definition extends Igl\'oi's \cite{Igloi} original formulation in the following way. Igl\'oi additionally assumes $\tim=[0,\infty)$, $X_0=0$, $X_1$ is non-Gaussian, and $X_t$ has finite moments of arbitrary order for every $t\geq0$ in which case he was able to show that the parameters $\alpha,\delta$ are uniquely determined and restricted to $\alpha>0$, $\delta\leq2\alpha$. We refuse to assume these additional conditions, since uniqueness of the parameters does not matter here. Roughly speaking, for $\delta\not=0$ dilative stability means that moving along the one-parameter semigroup $(\mu^s)_{s>0}$ generated by the finite-dimensional marginal distribution $\mu$ of $(X_{t_1},\ldots,X_{t_k})$ coincides with the distribution of the space-time transformation $s^{\frac12-\frac{\alpha}{\delta}}(X_{s^{1/\delta}t_1},\ldots,X_{s^{1/\delta}t_k})$, whereas for $\delta=0$ dilative stability coincides with selfsimilarity. Note that Kaj \cite{Kaj} introduced a weaker scaling relation called {\it aggregate-similarity}, which has been extended in Definition 1.4 of \cite{BPK} such that dilative stability and aggregate similarity essentially define the same property if one additionally assumes infinite divisibility and weak right-continuity of the finite-dimensional marginal distributions; see Proposition 1.5 in \cite{BPK} for details. 

In Section 2 we will introduce a weaker scaling property called dilative semistability which naturally comes into play assuming weak continuity. This notion extends the class of infinitely divisible semi-selfsimilar processes introduced in \cite{MaeSat1}. We give some examples of dilatively semistable process, in particular we point out how dilatively semistable generalized fractional L\'evy motions can be constructed from dilatively stable counterparts of \cite{BPK}. Finally, in Section 3 we show that in a general limit procedure for certain aggregation models, dilatively stable and dilatively semistable processes can be characterized as limit processes.

\section{Dilatively semistable processes}

Let $X=(X_t)_{t\in\tim}$ be a stochastic process on $\rr$ whose finite-dimensional marginal distributions are infinitely divisible. Inspired by Urbanik's decomposability group in \cite{Urb}, for $\alpha,\delta\in\rr$ we define the {\it dilative decomposability group} of $X$ by
$$D_X(\alpha,\delta)=\left\{c>0:\begin{array}{c}
\psi_{ct_1,\ldots,ct_k}(\theta_1,\ldots,\theta_k)=c^\delta\psi_{t_1,\ldots,t_k}(c^{\alpha-\delta/2}\theta_1,\ldots,c^{\alpha-\delta/2}\theta_k)\\[1ex]
\text{for all $k\in\nat$, $\theta_1,\ldots,\theta_k\in\rr$, and $t_1,\ldots,t_k\in\tim$}\phantom{mm}
\end{array}\right\},$$
where $\psi_{t_1,\ldots,t_k}$ again denotes the log-characteristic function of $(X_{t_1},\ldots,X_{t_k})$ and the notion ``group'' is justified as follows.
\begin{prop}\label{ddg}
If the finite-dimensional distributions of $X$ are weakly continuous then $D_X(\alpha,\delta)$ is a closed subgroup of $\G=((0,\infty),\cdot)$.
\end{prop}
\begin{proof}
If $b,c\in D_X(\alpha,\delta)$ we have
\begin{align*}
\psi_{bct_1,\ldots,bct_k}(\theta_1,\ldots,\theta_k) & =b^\delta\psi_{ct_1,\ldots,ct_k}(b^{\alpha-\delta/2}\theta_1,\ldots,b^{\alpha-\delta/2}\theta_k)\\
& =(bc)^\delta\psi_{t_1,\ldots,t_k}((bc)^{\alpha-\delta/2}\theta_1,\ldots,(bc)^{\alpha-\delta/2}\theta_k)
\end{align*}
showing that $bc\in D_X(\alpha,\delta)$. Hence $D_X(\alpha,\delta)$ is a subgroup of $\G$. If $c_n\in D_X(\alpha,\delta)$, $n\in\nat$, is a sequence with $c_n\to c>0$ then our assumption on weak continuity implies
\begin{align*}
\psi_{ct_1,\ldots,ct_k}(\theta_1,\ldots,\theta_k) & =\lim_{n\to\infty}\psi_{c_nt_1,\ldots,c_nt_k}(\theta_1,\ldots,\theta_k)\\
& =\lim_{n\to\infty}c_n^\delta\psi_{t_1,\ldots,t_k}(c_n^{\alpha-\delta/2}\theta_1,\ldots,c_n^{\alpha-\delta/2}\theta_k)\\
& =c^\delta\psi_{t_1,\ldots,t_k}(c^{\alpha-\delta/2}\theta_1,\ldots,c^{\alpha-\delta/2}\theta_k)
\end{align*}
showing that $c\in D_X(\alpha,\delta)$. Hence $D_X(\alpha,\delta)$ is a closed subgroup of $\G$.
\end{proof}
Since the only non-trivial closed subgroups of $\G$ are $\G$ itself (leading to dilative stability) and $c^\ganz=\{c^m:\,m\in\ganz\}$ for some $c>1$, the following property naturally appears.
\begin{defi}\label{dss}
A stochastic process $X=(X_t)_{t\in\tim}$ is said to be $(c,\alpha,\delta)$-{\it dilatively semistable} for parameters $c>1$ and $\alpha,\delta\in\rr$ if all of its finite-dimensional marginal distributions are infinitely divisible and $c^\ganz\subseteq D_X(\alpha,\delta)$.
\end{defi}
\begin{examples}
(a) By Definition \ref{dss}, any $(\alpha,\delta)$-dilatively stable process is also $(c,\alpha,\delta)$-dilatively semistable for every $c>1$.\\ Conversely, let $X=(X_t)_{t\in\tim}$ be a weakly continuous $(b,\alpha,\delta)$ and $(c,\alpha,\delta)$-dilatively semistable process, where $b,c>1$ are incommensurable in the sense that $b^n\not=c^m$ for all $n,m\in\ganz$. Then Proposition \ref{ddg} yields $(0,\infty)=\overline{\{b^nc^m:\,n,m\in\ganz\}}\subseteq D_X(\alpha,\delta)$ showing that $X$ is $(\alpha,\delta)$-dilatively stable.

(b) Let $X=(X_t)_{t\geq0}$ be a semi-selfsimilar process with Hurst index $H>0$, i.e.
$$(X_{ct})_{t\geq0}\eqfd(c^HX_t)_{t\geq0}\quad\text{ for some }c>1,$$
where ``$\eqfd$'' denotes equality in distribution of all finite-dimensional marginal distributions. Then obviously $X$ fulfills the scaling property of a $(c,H,0)$-dilatively semistable process for which (due to $\delta=0$) infinite divisibility is not needed. Hence dilative semistability extends semi-selfsimilarity for infinitely divisible processes.

(c) Let $X=(X_t)_{t\geq0}$ be a $(c,\gamma)$-semistable L\'evy process, i.e.\ a semi-selfsimilar L\'evy process with Hurst index $H=1/\gamma$ for some $c>1$ and $\gamma\in(0,2)$. Then by semi-selfsimilarity we have
\begin{align*}
\psi_{ct_1,\ldots,ct_k}(\theta_1,\ldots,\theta_k) & =\psi_{t_1,\ldots,t_k}(c^{1/\gamma}\theta_1,\ldots,c^{1/\gamma}\theta_k)
\intertext{and on the other hand for $\delta\in\ganz$ we get}
c^\delta\psi_{t_1,\ldots,t_k}(c^{\alpha-\delta/2}\theta_1,\ldots,c^{\alpha-\delta/2}\theta_1) & = \psi_{c^\delta t_1,\ldots,c^\delta t_k}(c^{\alpha-\delta/2}\theta_1,\ldots,c^{\alpha-\delta/2}\theta_k)\\
& =\psi_{t_1,\ldots,t_k}(c^{\alpha-\delta/2+\delta/\gamma}\theta_1,\ldots,c^{\alpha-\delta/2+\delta/\gamma}\theta_k),
\end{align*}
where the first equality is due to the fact that $X$ is a L\'evy process and the second equality follows from semi-selfsimilarity. If $\frac1{\gamma}=\alpha-\frac{\delta}{2}+\frac{\delta}{\gamma}$, i.e.\ $\alpha=\frac{1-\delta}{\gamma}+\frac{\delta}{2}$ this shows that $X$ is $(c,\frac{1-\delta}{\gamma}+\frac{\delta}{2},\delta)$-dilatively semistable for every $\delta\in\ganz$. In particular, the parameters are not uniquely determined.
\end{examples}
To give a more advanced example we now turn to the class of generalized fractional L\'evy processes, extending section 2 of \cite{BPK}. Let $(L_t^{(1)})_{t\geq0}$ be a centered L\'evy process without Gaussian component, whose L\'evy measure $\phi$ fulfills $\int_{\{|x|>1\}}x^2\,\phi(dx)<\infty$, so that $\Exp[(L_t^{(1)})^2]=t\cdot\Exp[(L_1^{(1)})^2]=t\int_{\rr}x^2\,\phi(dx)$. We now consider the two-sided L\'evy process $L=(L_t)_{t\in\rr}$ with
$$L_t=L_t^{(1)}\cdot 1_{[0,\infty)}(t)-L_{(-t)-}^{(2)}\cdot 1_{(-\infty,0)}(t),$$
where $L^{(2)}$ denotes an independent copy of $L^{(1)}$; cf.\ section 2 in \cite{KM}. Marquardt \cite{Mar} has shown that in this case for any Borel-measurable function $f:\rr^2\to\rr$ such that $u\mapsto f(t,u)$ belongs to $L^2(\rr)$ for all $t\in\rr$, the integral $X_t=\int_{\rr}f(t,u)\,L(du)$ exists in $L^2(\Omega,\mathcal A,P)$. Moreover, the characteristic function of $(X_{t_1},\ldots,X_{t_k})$ takes the form 
\begin{equation}\label{cfgLp}
\Exp\left[\exp\Big(\sum_{j=1}^k\theta_jX_{t_j}\Big)\right]=\exp\left(-\int_\rr\varphi\Big(\sum_{j=1}^k\theta_jf(t_j,u)\Big)du\right),
\end{equation}
where
$$\varphi(\theta)=\int_\rr(e^{i\theta x}-1-i\theta x)\,\phi(dx)$$
is the log-characteristic function of $L^{(1)}$. The process $X=(X_t)_{t\in\rr}$ is called a {\it generalized fractional L\'evy process} with kernel function $f$ according to \cite{KM} and it is shown in the proof of Proposition 2.3 in \cite{BPK} that $X$ is infinitely divisible.
\begin{prop}\label{condDSSgfLp}
If for some $c>1$ the kernel function $f$ satisfies
 \begin{equation}\label{DSS_funct_eq}
  f(ct,c^\delta u)= c^{\alpha-\frac{\delta}{2}} f(t,u) \quad\text{ for all } t,u\in\rr
 \end{equation}
 then the generalized fractional L\'evy process $(X_t)_{t\in\rr}$ is $(c,\alpha, \delta)$-dilatively semistable.
\end{prop}
\begin{proof}
By \eqref{cfgLp} the log-characteristic function of $(X_{t_1}, \ldots, X_{t_k})$ has the form
$$\psi_{t_1,\ldots,t_k}(\theta_1, \ldots, \theta_k)=-\int_\rr \varphi\Big(\sum_{j=1}^k \theta_j f(t_j, u)\Big)du$$
and hence using \eqref{DSS_funct_eq} and a change of variables $s=c^\delta u$ we get
 \begin{align*}
 & c^\delta\psi_{t_1,\ldots,t_k}(c^{\alpha-\frac{\delta}{2}} \theta_1, \ldots,c^{\alpha-\frac{\delta}{2}} \theta_k)
 = -c^\delta\int_\rr\varphi\Big(\sum_{j=1}^k c^{\alpha-\frac{\delta}{2}} \theta_j f(t_j, u)\Big)du\\
  & \quad= -c^\delta\int_\rr\varphi\Big(\sum_{j=1}^k \theta_j f(ct_j, c^\delta u)\Big)du
  = -\int_\rr \varphi\Big(\sum_{j=1}^k \theta_j f(ct_j, s)\Big)ds \\
  & \quad= \psi_{ct_1,\ldots,ct_k}(\theta_1, \ldots, \theta_k)
 \end{align*}
 showing that $c\in D_X(\alpha,\delta)$. By Proposition \ref{ddg} we get $c^\ganz\subseteq D_X(\alpha,\delta)$ which yields the assertion.
\end{proof}
\begin{remark}
In section 2 of \cite{BPK} explicit examples of generalized fractional L\'evy processes that are dilatively stable are given. By Proposition 2.3 in \cite{BPK}, a sufficient condition for dilative stability is that the kernel function fulfills the scaling relation
 \begin{equation}\label{DS_funct_eq}
  f(Tt,T^\delta u)= T^{\alpha-\frac{\delta}{2}} f(t,u) \quad\text{ for all } t,u\in\rr\text{ and }T>0,
 \end{equation}
which is slightly stronger than \eqref{DSS_funct_eq}. Note that for any $c>1$ and $\delta>0$ we can directly generate examples of dilatively semistable generalized fractional L\'evy processes (that are not dilatively stable) using the functions
$$f_c(t,u)=f\left(c^{\lfloor\log_ct\rfloor},c^{\delta\lfloor\log_{c^\delta}u\rfloor}\right)\quad \text{ for }t,u\in\rr,$$
where $f$ fulfills \eqref{DS_funct_eq}, provided that $f_c$ is still a valid kernel function. Indeed, by \eqref{DS_funct_eq} we have for all $t,u\in\rr$
\begin{align*}
f_c(ct,c^\delta u) & =f\left(c^{1+\lfloor\log_ct\rfloor},c^{\delta(1+\lfloor\log_{c^\delta}u\rfloor)}\right)\\
 & =c^{\alpha-\frac{\delta}{2}}f\left(c^{\lfloor\log_{c^\delta}u\rfloor},c^{\delta\lfloor\log_ct\rfloor}\right)=c^{\alpha-\frac{\delta}{2}}f_c(t,u)
\end{align*}
showing that $f_c$ fulfills \eqref{DSS_funct_eq}.
\end{remark}

\section{Dilative semistability as a property of limit processes}

By Lamperti's Theorem 2 in \cite{Lam}, it is well known that selfsimilar stochastic processes $X=(X_t)_{t\geq0}$ can be characterized by limit theorems of the form
\begin{equation}\label{lamperti}
f(T)Y_{Tt}\fdd X_t\quad\text{ as }T\to\infty
\end{equation}
for some stochastic process $Y=(Y_t)_{t\geq0}$ and a necessarily regularly varying normalization function $f:(0,\infty)\to(0,\infty)$, where ``$\fdd$'' denotes convergence of all finite-dimensional distributions. Igl\'oi extended this characterization to dilatively stable processes $X$ in Theorem 2.2.7 of \cite{Igloi} by additionally introducing a convolution exponent $g(T)$ for the process $Y$ in \eqref{lamperti}. This requires infinite divisibility of the process $Y$ and, since dilatively stable processes in the sense of Igl\'oi are non-Gaussian with $X_0=0$ and have finite moments of arbitrary order, additionally in \cite{Igloi} a corresponding convergence for all cumulants is required. As mentioned in the Introduction, in this case the parameters $\alpha,\delta$ of dilative stability are uniquely determined and restricted to $\alpha>0$, $\delta\leq 2\alpha$, so that Igl\'oi was able to show that the scaling functions $f,g$ are necessarily regularly varying. In our setting, the parameters $\alpha,\delta$ are not necessarily unique. Hence we will have to assume regular variation of the appropriate normalization sequences but, due to a formulation in terms of aggregation schemes, we do not have to require infinite divisibility or finite moment conditions for the process $Y$. Recall that a positive sequence $(a_n)_{n\in\nat}\subseteq(0,\infty)$ is called regularly varying of index $\gamma\in\rr$ if for any $\lambda>0$ we have
$$\frac{a_{\lfloor\lambda n\rfloor}}{a_n}\to\lambda^\gamma\quad\text{ as }n\to\infty$$
and this convergence automatically holds uniformly on compact intervals of $\{\lambda>0\}$; e.g., see Corollary 4.2.11 in \cite{MMMHPS}. For short we will write $(a_n)_{n\in\nat}\in\RV(\gamma)$ and in case $\gamma=0$ the sequence is also called slowly varying.
\begin{thm}\label{dslt}
(a) Assume that for some $\alpha,\delta\in\rr$ there exist regularly varying sequences $(a_n)_{n\in\nat}\in\RV(\frac{\delta}{2}-\alpha)$ and $(b_n)_{n\in\nat}\subseteq\nat$ with $(b_n)_{n\in\nat}\in\RV(|\delta|)$, where in case $\delta=0$ we additionally assume $b_n\to\infty$, such that for some stochastic processes $X=(X_t)_{t\in\tim}$, $Y=(Y_t)_{t\in\tim}$ with $X$ being weakly continuous we have that for every $k\in\nat$ and $(t_1,\ldots,t_k)\in\tim^k$ one of the following two conditions for i.i.d.\ copies $(Y^{(i)})_{i\in\nat}$ of $Y$ is fulfilled.
\begin{itemize}
\item[\it (a1)] If $\delta\leq0$ the convergence
\begin{equation}\label{ltdl0}
a_n\sum_{i=1}^{b_n}\left(Y_{nt_1}^{(i)},\ldots,Y_{nt_k}^{(i)}\right)\fdc (X_{t_1},\ldots,X_{t_k})
\end{equation}
holds uniformly on compact subsets of the time parameters $(t_1,\ldots,t_k)\in\tim^k$.
\item[\it (a2)] If $\delta\geq0$ the convergence
\begin{equation}\label{ltdg0}
a_n^{-1}\sum_{i=1}^{b_n}\left(Y_{t_1/n}^{(i)},\ldots,Y_{t_1/n}^{(i)}\right)\fdc (X_{t_1},\ldots,X_{t_k})
\end{equation}
holds uniformly on compact subsets of the time parameters $(t_1,\ldots,t_k)\in\tim^k$.
\end{itemize}
Then $X$ is $(\alpha,\delta)$-dilatively stable.

(b) Conversely, if $X=(X_t)_{t\in\tim}$ is a weakly continuous $(\alpha,\delta)$-dilatively stable process for some $\alpha,\delta\in\rr$ then \eqref{ltdl0} in case $\delta\leq0$, respectively \eqref{ltdg0} in case $\delta\geq0$, holds uniformly on compact subsets of the time parameters $(t_1,\ldots,t_k)\in\tim^k$ for the sequences $a_n=n^{\frac{\delta}{2}-\alpha}$ and $b_n=\lfloor n^{|\delta|}\rfloor$, where now $(Y^{(i)})_{i\in\nat}$ are i.i.d.\ copies of $X$. 
\end{thm}
\begin{remark}
Note that in case $\delta=0$ we additionally assume $b_n\to\infty$ for the slowly varying sequence $(b_n)_{n\in\nat}\subseteq\nat$ in part (a) in order to be able to conclude infinite divisibility of $X$. Since the case $\delta=0$ belongs to selfsimilar limit processes and a bounded sequence $(b_n)_{n\in\nat}\subseteq\nat$ has an eventually constant subsequence, in view of the corresponding result in Theorem 2 of \cite{Lam} the assumption $b_n\to\infty$ entails no loss of generality. The same remark holds true for semi-selfsimilar limit processes in case $\delta=0$ of Theorem \ref{dsslt}(a) below. The corresponding result for eventually constant sequences $(b_n)_{n\in\nat}\subseteq\nat$ is given by Theorem 2 in \cite{MaeSat1}.
\end{remark}
\begin{proof}[Proof of Theorem \ref{dslt}]
(a) We first consider the case $\delta\leq0$. Since $X$ is assumed to be weakly continuous, \eqref{ltdl0} is equivalent to
\begin{equation}\label{ltdl01}
a_n\sum_{i=1}^{b_n}\left(Y_{nt_1^{(n)}}^{(i)},\ldots,Y_{nt_k^{(n)}}^{(i)}\right)\fdc (X_{t_1},\ldots,X_{t_k})
\end{equation}
for all sequences $(t_1^{(n)},\ldots,t_k^{(n)})\to(t_1,\ldots,t_k)\in\tim^k$. Hence the distribution of $(X_{t_1},\ldots,X_{t_k})$ is infinitely divisible by Lemma 1.6.1(b) in \cite{HS}. Let $\psi_{t_1,\ldots,t_k}$ denote the log-characteristic function of $(X_{t_1},\ldots,X_{t_k})$ and let $\nu_{t_1,\ldots,t_k}$ be the characteristic function of $(Y_{t_1},\ldots,Y_{t_k})$. Then
by L\'evy's continuity theorem, \eqref{ltdl01} can be equivalently formulated as
\begin{equation}\label{ltdl0ft}
\left(\nu_{nt_1^{(n)},\ldots,nt_k^{(n)}}\big(a_n\theta_1^{(n)},\ldots,a_n\theta_k^{(n)}\big)\right)^{b_n}\to\exp\left(\psi_{t_1,\ldots,t_k}(\theta_1,\ldots,\theta_k)\right)
\end{equation}
for all sequences $(t_1^{(n)},\ldots,t_k^{(n)},\theta_1^{(n)},\ldots,\theta_k^{(n)})\to(t_1,\ldots,t_k,\theta_1,\ldots,\theta_k)\in\tim^k\times\rr^k$. Moreover, due to Lemma 1.6.1(a) in \cite{HS} we have $\nu_{nt_1^{(n)},\ldots,nt_k^{(n)}}\big(a_n\theta_1^{(n)},\ldots,a_n\theta_k^{(n)}\big)\to1$ and hence with the principal branch of the complex logarithm we get as $n\to\infty$
$$\log\left(\nu_{nt_1^{(n)},\ldots,nt_k^{(n)}}\big(a_n\theta_1^{(n)},\ldots,a_n\theta_k^{(n)}\big)\right)\sim\nu_{nt_1^{(n)},\ldots,nt_k^{(n)}}\big(a_n\theta_1^{(n)},\ldots,a_n\theta_k^{(n)}\big)-1.$$
Further, we have for sufficiently large $n\in\nat$
\begin{align*}
& \log\left(\left(\nu_{nt_1^{(n)},\ldots,nt_k^{(n)}}\big(a_n\theta_1^{(n)},\ldots,a_n\theta_k^{(n)}\big)\right)^{b_n}\right)\\
&\quad=b_n\log\left(\nu_{nt_1^{(n)},\ldots,nt_k^{(n)}}\big(a_n\theta_1^{(n)},\ldots,a_n\theta_k^{(n)}\big)\right)+2\pi i\,m_n
\end{align*}
for some sequence $m_n\in\ganz$. Hence, as $n\to\infty$ it follows by \eqref{ltdl0ft}
\begin{equation}\label{acceq}\begin{split}
& \exp\left(b_n\left(\nu_{nt_1^{(n)},\ldots,nt_k^{(n)}}\big(a_n\theta_1^{(n)},\ldots,a_n\theta_k^{(n)}\big)-1\right)\right)\\
& \quad\sim\exp\left(b_n\log\left(\nu_{nt_1^{(n)},\ldots,nt_k^{(n)}}\big(a_n\theta_1^{(n)},\ldots,a_n\theta_k^{(n)}\big)\right)\right)\\
& \quad=\exp\left(b_n\log\left(\nu_{nt_1^{(n)},\ldots,nt_k^{(n)}}\big(a_n\theta_1^{(n)},\ldots,a_n\theta_k^{(n)}\big)\right)+2\pi i\,m_n\right)\\
& \quad=\exp\left(\log\left(\left(\nu_{nt_1^{(n)},\ldots,nt_k^{(n)}}\big(a_n\theta_1^{(n)},\ldots,a_n\theta_k^{(n)}\big)\right)^{b_n}\right)\right)\\
& \quad=\left(\nu_{nt_1^{(n)},\ldots,nt_k^{(n)}}\big(a_n\theta_1^{(n)},\ldots,a_n\theta_k^{(n)}\big)\right)^{b_n}\\
& \quad\to\exp\left(\psi_{t_1,\ldots,t_k}(\theta_1,\ldots,\theta_k)\right)
\end{split}\end{equation}
for all sequences $(t_1^{(n)},\ldots,t_k^{(n)},\theta_1^{(n)},\ldots,\theta_k^{(n)})\to(t_1,\ldots,t_k,\theta_1,\ldots,\theta_k)\in\tim^k\times\rr^k$. Note that the left-hand side of \eqref{acceq} is the Fourier transform of an infinitly divisible compound Poisson distribution; e.g., see Definition 3.1.7 in \cite{MMMHPS}. Thus \eqref{acceq} is equivalent to
\begin{equation}\label{ltdl0acc}
b_n\left(\nu_{nt_1,\ldots,nt_k}(a_n\theta_1,\ldots,a_n\theta_k)-1\right)\to\psi_{t_1,\ldots,t_k}(\theta_1,\ldots,\theta_k)
\end{equation}
uniformly on compact subsets of $(t_1,\ldots,t_k,\theta_1,\ldots,\theta_k)\in\tim^k\times\rr^k$; e.g., see Lemma 3.1.10 in \cite{MMMHPS}. Hence for every $T>0$ we get
\begin{equation}\label{ltft1}
b_n\left(\nu_{nTt_1,\ldots,nTt_k}(a_n\theta_1,\ldots,a_n\theta_k)-1\right)\to\psi_{Tt_1,\ldots,Tt_k}(\theta_1,\ldots,\theta_k).
\end{equation}
On the other hand we have by \eqref{ltdl0acc} and regular variation
\begin{equation}\label{ltft2}\begin{split}
& b_n\left(\nu_{nTt_1,\ldots,nTt_k}(a_n\theta_1,\ldots,a_n\theta_k)-1\right)\\
& \quad=\frac{b_n}{b_{\lfloor nT\rfloor}}\,b_{\lfloor nT\rfloor}\left(\nu_{\lfloor nT\rfloor\frac{nTt_1}{\lfloor nT\rfloor},\ldots,\lfloor nT\rfloor\frac{nTt_k}{\lfloor nT\rfloor}}\left(a_{\lfloor nT\rfloor}\frac{a_n}{a_{\lfloor nT\rfloor}}\theta_1,\ldots,a_{\lfloor nT\rfloor}\frac{a_n}{a_{\lfloor nT\rfloor}}\theta_k\right)-1\right)\\
& \quad\to T^\delta\psi_{t_1,\ldots,t_k}\left(T^{\alpha-\frac{\delta}{2}}\theta_1,\ldots,T^{\alpha-\frac{\delta}{2}}\theta_k\right).
\end{split}\end{equation}
A comparison of \eqref{ltft1} and \eqref{ltft2} shows that $T\in D_X(\alpha,\delta)$ for any $T>0$ and thus $X$ is $(\alpha,\delta)$-dilatively stable.

In case $\delta\geq0$, similarly we get by \eqref{ltdg0} that the distribution of $(X_{t_1},\ldots,X_{t_k})$ is infinitely divisible and 
\begin{equation}\label{jhjh}
b_n\left(\nu_{t_1/n,\ldots,t_k/n}(a_n^{-1}\theta_1,\ldots,a_n^{-1}\theta_k)-1\right)\to\psi_{t_1,\ldots,t_k}(\theta_1,\ldots,\theta_k)
\end{equation}
holds uniformly on compact subsets of $(t_1,\ldots,t_k,\theta_1,\ldots,\theta_k)\in\tim^k\times\rr^k$.
Hence for every $T>0$ we get
\begin{equation*}
b_n\left(\nu_{Tt_1/n,\ldots,Tt_k/n}(a_n^{-1}\theta_1,\ldots,a_n^{-1}\theta_k)-1\right)\to\psi_{Tt_1,\ldots,Tt_k}(\theta_1,\ldots,\theta_k),
\end{equation*}
and on the other hand for $n>T$ we have by \eqref{jhjh} and regular variation
\begin{equation*}\begin{split}
& b_n\left(\nu_{Tt_1/n,\ldots,Tt_k/n}(a_n^{-1}\theta_1,\ldots,a_n^{-1}\theta_k)-1\right)\\
& \quad=\frac{b_n}{b_{\lfloor \frac{n}{T}\rfloor}}\,b_{\lfloor \frac{n}{T}\rfloor}\left(\nu_{t_1\frac{T\lfloor \frac{n}{T}\rfloor}{n}/\lfloor \frac{n}{T}\rfloor,\ldots,t_k\frac{T\lfloor \frac{n}{T}\rfloor}{n}/\lfloor \frac{n}{T}\rfloor}\left(a_{\lfloor \frac{n}{T}\rfloor}^{-1}\frac{a_{\lfloor \frac{n}{T}\rfloor}}{a_n}\theta_1,\ldots,a_{\lfloor \frac{n}{T}\rfloor}^{-1}\frac{a_{\lfloor \frac{n}{T}\rfloor}}{a_n}\theta_k\right)-1\right)\\
& \quad\to T^\delta\psi_{t_1,\ldots,t_k}(T^{\alpha-\frac{\delta}{2}}\theta_1,\ldots,T^{\alpha-\frac{\delta}{2}}\theta_k),
\end{split}\end{equation*}
showing again that $X$ is $(\alpha,\delta)$-dilatively stable.

(b) We have for $\delta\leq0$ using that $n\in D_X(\alpha,\delta)$ 
\begin{align*}
\lfloor n^{-\delta}\rfloor\psi_{nt_1,\ldots,nt_k}(n^{\frac{\delta}{2}-\alpha}\theta_1,\ldots,n^{\frac{\delta}{2}-\alpha}\theta_k) & =\lfloor n^{-\delta}\rfloor n^\delta\psi_{t_1,\ldots,t_k}(\theta_1,\ldots,\theta_k)\\
& \to\psi_{t_1,\ldots,t_k}(\theta_1,\ldots,\theta_k),
\end{align*}
and for $\delta\geq0$ using that $1/n\in D_X(\alpha,\delta)$
\begin{align*}
\lfloor n^{\delta}\rfloor\psi_{t_1/n,\ldots,t_k/n}(n^{\alpha-\frac{\delta}{2}}\theta_1,\ldots,n^{\alpha-\frac{\delta}{2}}\theta_k) & =\lfloor n^{\delta}\rfloor n^{-\delta}\psi_{t_1,\ldots,t_k}(\theta_1,\ldots,\theta_k)\\
& \to\psi_{t_1,\ldots,t_k}(\theta_1,\ldots,\theta_k).
\end{align*}
Since $X$ is weakly continuous, this shows that \eqref{ltdl0} and \eqref{ltdg0} hold uniformly on compact subsets of $(t_1,\ldots,t_k)\in\tim^k$ with the proposed choices of sequences $(a_n)_{n\in\nat}$ and $(b_n)_{n\in\nat}$ and with i.i.d.\ copies $(Y^{(i)})_{i\in\nat}$ of $X$.
\end{proof}
\begin{example}
An explicit example of a limit theorem of the form \eqref{ltdl0} is given by Pilipauskait\.{e} and Surgailis \cite{PS}. They consider the aggregation of
$$Y_t^{(i)}=\sum_{k=1}^{\lfloor t\rfloor}X_i(k),\quad t\geq0,$$
for certain i.i.d.\ stationary random coefficient AR(1) processes $(X_i)_{i\in\nat}$, where the random coefficient depends on a parameter $\beta\in(-1,1)$. In Theorem 2.2 of \cite{PS} it is particularly shown that for any $\beta\in(-1,1)$, $k\in\nat$ and $(t_1,\ldots,t_k)\in\rr_+^k$
\begin{equation}\label{exPS} 
n^{-3/2}\sum_{i=1}^{\lfloor n^{1+\beta}\rfloor}\left(Y_{nt_1}^{(i)},\ldots,Y_{nt_k}^{(i)}\right)\fdc (Z_\beta(t_1),\ldots,Z_\beta(t_k)),
\end{equation}
where the limit process $Z_\beta=(Z_\beta(t))_{t\geq0}$ is infinitely divisible by Proposition 3.1 in \cite{PS} and given by the log-characteristic function
\begin{align*}
& \psi_{t_1,\ldots,t_k}(\theta_1,\ldots,\theta_k)\\
& \quad=C \int_0^\infty\bigg(\exp\bigg\{-\frac{1}{2}\int_\rr \Big(\sum_{j=1}^k\theta_j\,{\textstyle\frac{\big(1-e^{(s-t_j)x}\big)1_{\{s<t_j\}}-\big(1-e^{sx}\big)1_{\{s<0\}}}{x}}\Big)^2\dd s\bigg\} - 1\bigg)x^\beta\,\dd x
\end{align*}
for some constant $C>0$. The process $Z_\beta$ is already known to be $(1-\beta/2,-1-\beta)$-dilatively stable by Proposition 3.1 in \cite{BPK}. Note that $Z_\beta$ is weakly continuous which follows easily by dominated convergence applied to the above log-characteristic function. Hence, dilative stability of $Z_\beta$ also follows from our Theorem 3.1(a), provided that the convergence in \eqref{exPS} is uniformly on compact subsets of $(t_1,\ldots,t_k)\in\rr_+^k$. Due to the lengthy derivation of \eqref{exPS} in \cite{PS} we renounce to check this in detail.

A further example might be deduced from Theorem 2 in \cite{GaiKaj}, where it is known from section 3 of \cite{BPK} that the limit process $Y_\beta$ is $((3-\beta)/2,1-\beta)$-dilatively stable for any parameter $\beta\in(1,2)$, but the limit theorem presented in Theorem 2 of \cite{GaiKaj} is not precisely of the form \eqref{ltdl0}.
\end{example}
We finally turn to a generalization of Theorem \ref{dslt} for dilatively semistable stochastic processes.
\begin{thm}\label{dsslt}
(a) Assume that for some $\alpha,\delta\in\rr$ there exist regularly varying sequences $(a_n)_{n\in\nat}\in\RV(\frac{\delta}{2}-\alpha)$ and $(b_n)_{n\in\nat}\subseteq\nat$ with $(b_n)_{n\in\nat}\in\RV(|\delta|)$, where in case $\delta=0$ we additionally assume $b_n\to\infty$, such that for some deterministic sequence $(k(n))_{n\in\nat}\subseteq\nat$ with $k(n+1)/k(n)\to c>1$ and some stochastic processes $X=(X_t)_{t\in\tim}$, $Y=(Y_t)_{t\in\tim}$ with $X$ being weakly continuous we have
that for every $k\in\nat$ and $(t_1,\ldots,t_k)\in\tim^k$ one of the following two conditions for i.i.d.\ copies $(Y^{(i)})_{i\in\nat}$ of $Y$ is fulfilled.
\begin{itemize}
\item[\it (a1)] If $\delta\leq0$ the convergence
\begin{equation}\label{ltdl0ss}
a_{k(n)}\sum_{i=1}^{b_{k(n)}}\left(Y_{k(n)t_1}^{(i)},\ldots,Y_{k(n)t_k}^{(i)}\right)\fdc (X_{t_1},\ldots,X_{t_k})
\end{equation}
holds uniformly on compact subsets of the time parameters $(t_1,\ldots,t_k)\in\tim^k$.
\item[\it (a2)] If $\delta\geq0$ the convergence
\begin{equation}\label{ltdg0ss}
a_{k(n)}^{-1}\sum_{i=1}^{b_{k(n)}}\left(Y_{t_1/k(n)}^{(i)},\ldots,Y_{t_1/k(n)}^{(i)}\right)\fdc (X_{t_1},\ldots,X_{t_k})
\end{equation}
holds uniformly on compact subsets of the time parameters $(t_1,\ldots,t_k)\in\tim^k$.
\end{itemize}
Then $X$ is $(c,\alpha,\delta)$-dilatively semistable.

(b) Conversely, if $X=(X_t)_{t\in\tim}$ is a weakly continuous $(c,\alpha,\delta)$-dilatively semistable process for some $c>1$ and $\alpha,\delta\in\rr$ then \eqref{ltdl0ss} in case $\delta\leq0$, respectively \eqref{ltdg0ss} in case $\delta\geq0$, holds uniformly on compact subsets of $(t_1,\ldots,t_k)\in\tim^k$ for the sequences $a_n=n^{\frac{\delta}{2}-\alpha}$,  $b_n=\lfloor n^{|\delta|}\rfloor$ and $k(n)=\lfloor c^n\rfloor$, where now $(Y^{(i)})_{i\in\nat}$ are i.i.d.\ copies of $X$.
\end{thm}
\begin{proof}
(a) As in the proof of Theorem \ref{dslt} $X$ is infinitely divisible and it follows from the weak continuity of $X$ and \eqref{ltdl0ss} that in case $\delta\leq0$
\begin{equation}\label{ltdl0accss}
b_{k(n)}\left(\nu_{k(n)t_1,\ldots,k(n)t_k}(a_{k(n)}\theta_1,\ldots,a_{k(n)}\theta_k)-1\right)\to\psi_{t_1,\ldots,t_k}(\theta_1,\ldots,\theta_k)
\end{equation}
uniformly on compact subsets of $(t_1,\ldots,t_k,\theta_1,\ldots,\theta_k)\in\tim^k\times\rr^k$. Hence we get
\begin{align*}
& b_{k(n)}\left(\nu_{k(n+1)t_1,\ldots,k(n+1)t_k}(a_{k(n)}\theta_1,\ldots,a_{k(n)}\theta_k)-1\right)\\
& \quad=b_{k(n)}\left(\nu_{k(n)\,\frac{k(n+1)}{k(n)}t_1,\ldots,k(n)\,\frac{k(n+1)}{k(n)}t_k}(a_{k(n)}\theta_1,\ldots,a_{k(n)}\theta_k)-1\right)\\
& \quad\to\psi_{ct_1,\ldots,ct_k}(\theta_1,\ldots,\theta_k).
\end{align*}
On the other hand we have by \eqref{ltdl0accss} and regular variation
\begin{align*}
& b_{k(n)}\left(\nu_{k(n+1)t_1,\ldots,k(n+1)t_k}(a_{k(n)}\theta_1,\ldots,a_{k(n)}\theta_k)-1\right)\\
& \,=\frac{b_{k(n)}}{b_{k(n+1)}}\,b_{k(n+1)}\left(\nu_{k(n+1)t_1,\ldots,k(n+1)t_k}\left(a_{k(n+1)}\frac{a_{k(n)}}{a_{k(n+1)}}\theta_1,\ldots,a_{k(n+1)}\frac{a_{k(n)}}{a_{k(n+1)}}\theta_k\right)-1\right)\\
& \,\to c^\delta\psi_{t_1,\ldots,t_k}\left(c^{\alpha-\frac{\delta}{2}}\theta_1,\ldots,c^{\alpha-\frac{\delta}{2}}\theta_k\right),
\end{align*}
showing that $c\in D_X(\alpha,\delta)$ and thus $X$ is $(c,\alpha,\delta)$-dilatively semistable.

In case $\delta\geq0$, similarly we get by \eqref{ltdg0ss}
\begin{equation}\label{ltdl0accss1}
b_{k(n)}\left(\nu_{t_1/k(n),\ldots,t_k/k(n)}(a_{k(n)}^{-1}\theta_1,\ldots,a_{k(n)}^{-1}\theta_k)-1\right)\to\psi_{t_1,\ldots,t_k}(\theta_1,\ldots,\theta_k)
\end{equation}
uniformly on compact subsets of $(t_1,\ldots,t_k,\theta_1,\ldots,\theta_k)\in\tim^k\times\rr^k$. Hence we get
\begin{align*}
& b_{k(n+1)}\left(\nu_{t_1/k(n),\ldots,t_k/k(n)}(a_{k(n+1)}^{-1}\theta_1,\ldots,a_{k(n+1)}^{-1}\theta_k)-1\right)\\
& \quad=b_{k(n+1)}\left(\nu_{\frac{k(n+1)}{k(n)}t_1/k(n),\ldots,\frac{k(n+1)}{k(n)}t_k/k(n)}(a_{k(n+1)}^{-1}\theta_1,\ldots,a_{k(n+1)}^{-1}\theta_k)-1\right)\\
& \quad\to\psi_{ct_1,\ldots,ct_k}(\theta_1,\ldots,\theta_k),
\end{align*}
and on the other hand for $n>T$ we have by \eqref{ltdl0accss1} and regular variation
\begin{align*}
& b_{k(n+1)}\left(\nu_{t_1/k(n),\ldots,t_k/k(n)}(a_{k(n+1)}^{-1}\theta_1,\ldots,a_{k(n+1)}^{-1}\theta_k)-1\right)\\
& \,=\frac{b_{k(n+1)}}{b_{k(n)}}\,b_{k(n)}\left(\nu_{t_1/k(n),\ldots,t_k/k(n)}\left(a_{k(n)}^{-1}\frac{a_{k(n)}}{a_{k(n+1)}}\theta_1,\ldots,a_{k(n)}^{-1}\frac{a_{k(n)}}{a_{k(n+1)}}\theta_k\right)-1\right)\\
& \,\to c^\delta\psi_{t_1,\ldots,t_k}\left(c^{\alpha-\frac{\delta}{2}}\theta_1,\ldots,c^{\alpha-\frac{\delta}{2}}\theta_k\right),
\end{align*}
showing again that $X$ is $(c,\alpha,\delta)$-dilatively semistable.

(b) We have for $\delta\leq0$ using that $c^n\in D_X(\alpha,\delta)$ 
\begin{align*}
& \lfloor\lfloor c^n\rfloor^{-\delta}\rfloor\psi_{\lfloor c^n\rfloor t_1,\ldots,\lfloor c^n\rfloor t_k}(\lfloor c^n\rfloor^{\frac{\delta}{2}-\alpha}\theta_1,\ldots,\lfloor c^n\rfloor^{\frac{\delta}{2}-\alpha}\theta_k)\\ 
& \quad=\lfloor\lfloor c^n\rfloor^{-\delta}\rfloor c^{n\delta}\psi_{\frac{\lfloor c^n\rfloor}{c^n}t_1,\ldots,\frac{\lfloor c^n\rfloor}{c^n}t_k}\left(\left(\frac{\lfloor c^n\rfloor}{c^n}\right)^{\frac{\delta}{2}-\alpha}\theta_1,\ldots,\left(\frac{\lfloor c^n\rfloor}{c^n}\right)^{\frac{\delta}{2}-\alpha}\theta_k\right)\\
& \quad\to\psi_{t_1,\ldots,t_k}(\theta_1,\ldots,\theta_k),
\end{align*}
and for $\delta\geq0$ using that $c^{-n}\in D_X(\alpha,\delta)$ 
\begin{align*}
& \lfloor\lfloor c^n\rfloor^{\delta}\rfloor\psi_{t_1/\lfloor c^n\rfloor,\ldots,t_k/\lfloor c^n\rfloor}(\lfloor c^n\rfloor^{\alpha-\frac{\delta}{2}}\theta_1,\ldots,\lfloor c^n\rfloor^{\alpha-\frac{\delta}{2}}\theta_k)\\ 
& \quad=\lfloor\lfloor c^n\rfloor^{\delta}\rfloor c^{-n\delta}\psi_{\frac{c^n}{\lfloor c^n\rfloor}t_1,\ldots,\frac{c^n}{\lfloor c^n\rfloor}t_k}\left(\left(\frac{\lfloor c^n\rfloor}{c^n}\right)^{\alpha-\frac{\delta}{2}}\theta_1,\ldots,\left(\frac{\lfloor c^n\rfloor}{c^n}\right)^{\alpha-\frac{\delta}{2}}\theta_k\right)\\
& \quad\to\psi_{t_1,\ldots,t_k}(\theta_1,\ldots,\theta_k),
\end{align*}
showing that \eqref{ltdl0ss} and \eqref{ltdg0ss} hold uniformly on compact subsets of $(t_1,\ldots,t_k)\in\tim^k$ with the proposed choices of sequences $(a_n)_{n\in\nat}$, $(b_n)_{n\in\nat}$ and $(k(n))_{n\in\nat}$ and with i.i.d.\ copies $(Y^{(i)})_{i\in\nat}$ of $X$.
\end{proof}
\noindent {\bf Acknowledgement.} We are grateful to M\'aty\'as Barczy and Gyula Pap for stimulating and helpful discussions on earlier versions of this manuscript.

\bibliographystyle{plain}

\begin{thebibliography}{00}

\bibitem{BPK} Barczy, M.; Kern, P.; and Pap, G. (2014) Dilatively stable stochastic processes and aggregate similarity. {\it Aequat. Math.} (to appear). {\tt http://arxiv.org/abs/1408.3919}

\bibitem{GaiKaj} Gaigalas, R.; and Kaj, I. (2003) Convergence of scaled renewal processes and a packet arrival model. {\it Bernoulli} {\bf 9} 671--703.

\bibitem{HS} Hazod, W.; and Siebert, E. (2001) {\it Stable Probability Measures on Euclidean Spaces and on Locally Compact Groups.} Kluwer, Dordrecht.

\bibitem{Igloi} Igl\'oi, E. (2008) {\it Dilative Stability.} PhD Thesis, Faculty of Informatics, University of Debrecen, Hungary. {\tt http://www.inf.unideb.hu/valseg/dolgozok/igloi/dissertation.pdf}

\bibitem{Kaj} Kaj, I. (2005) Limiting fractal random processes in heavy-tailed systems. In: J. Levy-Vehel, E. Lutton (eds.) {\it Fractals in Engineering, New Trends in Theory and Applications.} Springer, London, pp. 199--218.

\bibitem{KM} Kl\"uppelberg, C.; and Matsui, M. (2014) Generalized fractional L\'evy processes with fractional Brownian motion limit. (Preprint) {\tt http://mediatum.ub.tum.de/doc/1156351/1156351.pdf}

\bibitem{Lam} Lamperti, J.W. (1962) Semi-stable stochastic processes. {\it Trans. Amer. Math. Soc.} {\bf 104} 62--78.

\bibitem{MaeSat1} Maejima, M.; and Sato, K.I. (1999) Semi-selfsimilar processes. {\it J. Theoret. Probab.} {\bf 12} 347--373.

\bibitem{Mar} Marquardt, T. (2006) Fractional L\'evy processes with an application to long memory moving average processes. {\it Bernoulli} {\bf 12} 1099--1126. 

\bibitem{MMMHPS} Meerschaert, M.M.; and Scheffler, H.-P. (2001) {\it Limit Distributions for Sums of Independent Random Vectors.} Wiley, New York.

\bibitem{PS} Pilipauskait\.{e}, V.; and Surgailis, D. (2014) Joint temporal and contemporaneous aggregation of random-coefficient AR(1) processes. {\it Stoch. Process. Appl.} {\bf 124} 1011--1035.

\bibitem{Urb} Urbanik, K. (1972) L\'evy's probability measures on Euclidean spaces. {\it Studia Math.} {\bf 44} 119--148.

\end{thebibliography}

\end{document}